\documentclass[reqno,11pt]{amsart}

\usepackage{mathrsfs, amssymb, latexsym, amsmath, graphicx, psfrag}

\newtheorem{thm}{Theorem}

\newtheorem{cor}[thm]{Corollary}
\newtheorem{lemma}[thm]{Lemma}

\theoremstyle{remark}
\newtheorem{rmk}{Remark}[section]

\newcommand{\fsl}{\mathfrak{sl}_2}
\newcommand{\h}{\mathrm{h}}
\newcommand{\BC}{\mathbb{C}}
\newcommand{\BN}{\mathbb{N}}
\newcommand{\BR}{\mathbb{R}}

\newcommand{\BZ}{\mathbb{Z}}
\newcommand{\U}{U_\zeta}
\newcommand{\B}{B_\zeta}
\newcommand{\D}{D_\zeta}
\newcommand{\cD}{\mathscr{D}}
\newcommand{\Gz}{\Gamma_\zeta}
\newcommand{\pz}{\psi_\zeta}
\newcommand{\hpz}{\hat{\psi}_\zeta}
\newcommand{\tz}{\tau_\zeta}
\newcommand{\ev}{\mathrm{ev}}
\renewcommand{\t}{\theta}
\newcommand{\bU}{{\U^\ev}}
\newcommand{\hU}{\hat U_\zeta}
\newcommand{\hUe}{{\hat{U}^\ev_\zeta}}
\newcommand{\bhUe}{\bar{\hat U}^\ev_\zeta}
\newcommand{\bbU}{\bar U_\zeta^\ev}
\newcommand{\Z}{\mathfrak{Z}}
\newcommand{\bZ}{\Z^\ev}
\newcommand{\bbZ}{\tilde{\Z}^\ev}
\renewcommand{\r}{\mathbf{r}}
\newcommand{\tr}{\mathrm{tr}}
\newcommand{\Inv}{\mathrm{Inv}}
\newcommand{\ad}{\mathrm{ad}}
\newcommand{\Id}{\mathrm{Id}}
\newcommand{\diag}{\mathrm{diag}}
\newcommand{\qt}[1]{\tr_q^{#1}}
\newcommand{\p}{\mathfrak{p}}
\newcommand{\q}{\mathfrak{q}}

\title[quantum invariants]{Quantum invariants of 3-manifolds associated to restricted quantum groups}
\author{Qi Chen,\ \ Chih-Chien Yu\ \  and\ \  Yu Zhang}

\address{Department of Mathematics\newline\indent
Winston-Salem State University\newline\indent 
Winston Salem, NC 27110, USA}
\email{chenqi@wssu.edu}
\address{Department of Mathematics\newline\indent
University of Arkansas at Fort Smith\newline\indent  
Fort Smith, AR 72913, USA}
\email{lyu@uafortsmith.edu} 
\address{Department of Mathematics\newline\indent
University at Buffalo\newline\indent  
The State University of New York\newline\indent  
Buffalo, NY 14260, USA}
\email{yz26@buffalo.edu}

\begin{document}

\begin{abstract}
We show that the Witten-Reshetikhin-Turaev $SU(2)$ invariant and the Hennings invariant associated to the restricted quantum $\fsl$ are essentially the same for rational homology 3-spheres.
\end{abstract}

\maketitle
\addtocounter{footnote}{1}

\section{Introduction}

After the discovery of the Jones polynomial, Witten proposed in \cite{w}
an invariant of 3-manifolds using the Chern-Simons theory.
The first `mathematically rigorous' construction of this invariant was obtained by Reshetikhin and Turaev in 
\cite{rt} using the representation theory of the quantum group $\U(\fsl)$ at a root of unity $\zeta$. This
invariant, denoted $\tz$, is now known as the Witten-Reshetikhin-Turaev $SU(2)$ invariant, or the WRT $SU(2)$ invariant in short. On the other hand Hennings
showed in \cite{hennings} that one can define a 3-manifold invariant $\pz$, called the Hennings invariant,
independent of the representation theory. In this note we will 
show that these two invariants are essentially the same. More precisely we have

\begin{thm}\label{mt}
Let $M$ be a closed 3-manifold and $\zeta$ a root of unity of order $\ell>1$ then
\begin{equation}
\pz(M) = \h(M) \tz(M)\ ,
\end{equation}
where $\h(M)$ is the order of the first homology group if it is finite and zero otherwise.
\end{thm}

This theorem will be proved in Section \ref{main}.
The same relation was shown in \cite{cks} for the WRT $SO(3)$ invariant and the Hennings invariant associated to the
small quantum $\fsl$. In general the restricted quantum groups are harder to deal with than the corresponding small
ones. One new obstacle in the proof of the $SU(2)$ case is to show Lemma \ref{l1}. We will
also give a much simplified proof of a key lemma, Lemma 8 in \cite{cks}. 

\begin{rmk}\label{k}
Kauffman and Radford compared $\pz$ and $\tz$ in \cite{kr} when $\ell=8$ and $M$ is the Lens space $L(k,1)$.
Their calculation confirms the above theorem. 
Note that in their corollary on page 154 $INV(L(k,1))$ should be equal to $|k|$.
\end{rmk}

As a consequence of the above theorem we have the following corollary.

\begin{cor}\label{cor}
If $(\ell, \h(M))=1$ then the WRT $SU(2)$ invariant $\tz(M)$ is always an algebraic integer.
\end{cor}

\begin{proof}[Sketch of Proof]
Let $\B$ be the Borel subalgebra of $\U(\fsl)$. Its quantum double
$\D := \cD(\B)$ is a ribbon Hopf algebra, c.f. \cite{kr1}. 
Let $\hpz$ be the Hennings invariant constructed from $\D$.
It is proved in \cite{ck} that $\hpz(M)$ is an algebraic integer for any closed 3-manifold $M$.
Furthermore $\hpz(M) = \eta\,\pz(M)$ for some eighth root of unity. Therefore $\pz(M)$ is also an algebraic
integer. Theorem \ref{mt} then says that the only possible denominator for $\tz(M)$ is $\h(M)$. But according to
\cite{km} the only possible denominators for $\tz(M)$ are factors of $\ell$. Since $\ell$ and $M$ are coprime
we have $\tz(M)$ is an algebraic integer.
\end{proof}

\begin{rmk}
Special cases of the corollary were studied extensively by many authors. Murakami proved it in \cite{mu} when
$\ell$ is a prime. His proof was simplified by Masbaum and Roberts in \cite{mr} using the Kauffman Bracket.
Later Habiro proved it in \cite{habiro4} for all integral homology 3-spheres without any restriction on $\ell$.
Integrality is a very important property for quantum invariants because 
integral invariants can be used to extract topological information, c.f. 
\cite{gkp, chenle1}, and to construct TQFTs over Dedekind domains, c.f. \cite{g, chenle2}.
We believe that some similar relation between the WRT and Hennings invariants exists for higher
ranked quantum groups. Our proof of the integrality can then be used to prove the integrality for all
WRT invariants.
\end{rmk}

\noindent{\em Acknowledgment.} The first author would like to thank Thang Le for his help.

\section{The restricted quantum group}

Fix a root of unity $\zeta$ of order $\ell>1$. To simplify the arguments we will just take $\zeta = e^{2\pi i/\ell}$
and $\zeta^a = e^{2\pi i a/\ell}$. Also set $\t = \zeta^{1/2}$.
The restricted quantum group $\U = \U(\fsl)$ is a $\BC$-algebra generated by $E, F, K$ and $K^{-1}$ with
the relations:
$$
K K^{-1} = K^{-1}K = 1, \quad K E = \t E K,\quad K F = \t^{-1} F K, 
$$
$$
EF - FE = \frac{K^2-K^{-2}}{\t-\t^{-1}}
$$
and
\begin{equation}\label{ell}
E^\ell = F^\ell = 0, \quad K^{4\ell} = 1\ .
\end{equation}
It is a Hopf algebra with the comultiplication $\Delta$, antipode $S$ and counit $\epsilon$ given by
$$
\Delta(K) = K\otimes K,\quad \Delta(E) = 1\otimes E + E\otimes K^2,\quad \Delta(F) = F\otimes 1 + K^{-2}\otimes F,
$$
$$
S(K) = K^{-1},\quad S(E) = -EK^{-2}, \quad S(F) = -K^2 F,
$$
$$
\epsilon(K) = 1, \quad \epsilon(E) = \epsilon(F) = 0.
$$
For $i\in\BZ/4\ell$ let
$$
\pi_i := \sum_{j=1}^{4\ell} \t^{ij} K^j.
$$
Then one has
$$
K \pi_i = \t^{-i} \pi_i, \quad E\pi_i = \pi_{i-1}E,\quad F\pi_i = \pi_{i+1}F.
$$

Recall that an element $x$ in a Hopf algebra $H$ is said to be a {\em cointegral} if
$xy = yx = \epsilon(y) x$, $\forall y \in H$. An element $f$ in $H^*$ is said to be a {\em left integral}
if $g f = g(1) f$, $\forall g\in H^*$. It is known that $\U$ contains a nonzero cointegral
$$
\Lambda = F^{\ell-1}\pi_{\ell-1}E^{\ell-1}\ \ ,
$$
and a nonzero left integral
\begin{equation}\label{lin}
\lambda(F^i K^j E^m) = \delta_{i,\ell-1}\delta_{j,2(\ell-1)}\delta_{m,\ell-1}\ \ .
\end{equation}

Denote the quantum integer by $[n] = (\t^n-\t^{-n})/(\t-\t^{-1})$ and
the quantum factorial by $[n]!=[n] [n-1] \cdots [1]$.
The Hopf algebra $\U$ is quasi-triangular with the universal $R$-matrix in $\U\otimes \U$
$$
R = D\sum_{n=0}^{\ell-1} \frac{(\t-\t^{-1})^n}{[n]!}\t^{\frac{n(n-1)}2} E^n\otimes F^n \ .
$$
Here $D$ is the diagonal part with
$$
D = \frac1{4\ell}\sum_{m,n=0}^{4\ell-1} \t^{-\frac{mn}2} K^m\otimes K^n\ \ .
$$
Furthermore $\U$ is a ribbon Hopf algebra whose ribbon element $\r$ and its inverse $\r^{-1}$, see (\ref{eqr}) below, belong to the Hopf subalgebra $\bU$ generated by $E, F$ and $K^2$. 
We follow \cite{feigin} to describe the center $\bZ$ of $\bU$. Note that their $K$ and $\q$
are equal to our $K^2$ and $\t$ respectively.
The dimension of $\bZ$ is $3\ell - 1$ with a basis $e_i, w^\pm_j$, $0\le i\le\ell$
and $1\le j\le \ell-1$. These elements satisfy
\begin{equation}\label{eqc}
e_i e_j = \delta_{ij}e_i,\quad e_i w^\pm_j = \delta_{ij} w^\pm_j, \quad w^\pm_iw^\pm_j = w^\pm_iw^\mp_j = 0.
\end{equation}
To simplify notation we will consider $w^\pm_0=w^\pm_\ell = 0$.
The ribbon element $\r$ and its inverse $\r^{-1}$ are in $\bZ$:
\begin{equation}\label{eqr}
\r^{\pm 1} = \sum_{m=0}^\ell a_{\pm,m} e_m + b_{\pm,m} w_m + c_{\pm,m} w_m^+
\end{equation}
where $w_m := w_m^+ + w_m^-$ and
$$
a_{\pm,m} = (-1)^{m+1} \t^{\pm\frac{1-m^2}2}, 
$$
$$
b_{\pm,m} = \pm(-1)^{\ell-1} (\t-\t^{-1}) \t^{\pm\frac{1-m^2}2} \frac{m}{[m]},\quad
c_{\pm,m} = -\frac{\ell\ b_{\pm,m}}m \ . 
$$

\section{The universal invariant of bottom tangles}\label{bt}
The definition of the universal invariant is the same as Section 2.2 in \cite{cks}. 
See also \cite{hennings, habiro1, o3}. We include it here for ease of reading. 
A {\em bottom tangle} is an oriented framed tangle properly embedded in $\BR^2\times[0,1)$
such that its $i$-th component starts from $(0, 2i, 0)$ and ends at $(0, 2i-1, 0)$. Note that bottom tangles do not
have circle components. Bottom tangles are considered equivalent up to ambient isotopy relative to boundary. The
universal $\U$-invariant $\Gz$ of bottom tangles can be calculated as follows. Let $T$ be a bottom tangle. Choose
a generic diagram of $T$ and label it according to Figure \ref{f1}, where $R = \sum R_1\otimes R_2$ is the universal
$R$-matrix and $S$ is the antipode.
\begin{figure}[t]
\centering
\psfrag{1}{\hspace{0ex}$R_1$}
\psfrag{2}{\hspace{-1ex}$R_2$}
\psfrag{3}{\hspace{-4ex}$S(R_1)$}
\psfrag{4}{\hspace{-1ex}$R_2$}
\psfrag{5}{$K^{-2}$}
\psfrag{6}{$K^2$}
\includegraphics[width=\textwidth]{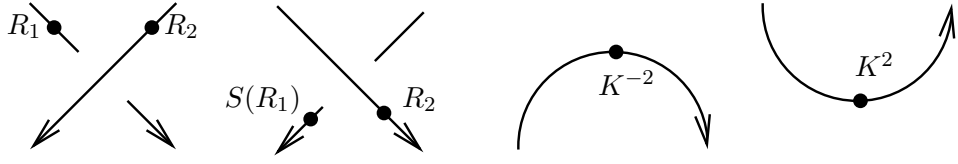}
\caption{Basic diagram labels.}\label{f1}
\end{figure}
Multiply the labels on each component, opposite to the orientation, to obtain an element in $\U$. The tensor
product of the elements from all components is the universal invariant $\Gz(T)$. Clearly if $T$ has $m$-component then
$\Gz(T)$ is in $\U^{\otimes m}$. But one can say a little more about the value of $\Gz(T)$.
Let's first recall that for a Hopf algebra $A$ and an $A$-module $W$, the invariant submodule $\Inv(W)$ is equal to
$\{ w \in W ~|~ a(w) = \epsilon(a) w, \forall a\in A\}$. The adjoint action makes $A$ to be an $A$-module, i.e.
$\ad_a(b) = \sum_{(a)} a'bS(a'')$.

\begin{lemma}\label{l1}
Let $T$ be an $m$-component bottom tangle such that $\hat T$ has 0 linking matrix. 
Then $\Gz(T) \in \Inv((\bU)^{\otimes m})$.
\end{lemma}

\begin{proof}
By Corollary 12 in \cite{kerler3}, $\Gz(T)\in \Inv((\U)^{\otimes m})$. 
So it is enough to show $\Gz(T)\in (\bU)^{\otimes m}$. This will follow from \cite{habiro1} Corollary 9.15, i.e. we
need to show that
\begin{enumerate}
	\item[(a)] if $x\otimes y$ is in $(\bU)^{\otimes 2}$ then so are $\sum\ad_{R_2}(y)\otimes \ad_{R_1}(x)$ and 		
	$\sum\ad_{S(R_1)}(y)\otimes \ad_{R_2}(x)$;
	\item[(b)] if $x$ is in $\bU$ then so are $\sum R_2S(\ad_{R_1}(x))$
	and $\sum S^{-1}(\ad_{R_1}(x))R_2$;
	\item[(c)] if $x$ is in $\bU$ then $\sum x'S(R_2)\otimes(\ad_{R_1}(x''))$ is in $(\bU)^{\otimes 2}$ and
	\item[(d)] $\Gz(B)$ is in $(\bU)^{\otimes 3}$ where $B$'s natural closure is the Borromean ring.
\end{enumerate}

These four statements follow from direct calculation. We will only show detail for (c, d). Set
$|F^mK^nE^p|=p-m$. Let $x\in \bU$,
\begin{align}\label{d}
&\quad \sum x'S(R_2)\otimes \ad_{R_1}(x'') \nonumber\\
&= \frac1{4\ell}\sum_{\substack{m,n,j=0\\ (x)}}^{4\ell-1}\t^{\frac{m(m-1)-nj}2}
\frac{(\t-\t^{-1})^m}{[m]!} x'S(K^jF^m)\otimes \ad_{K^n E^m}(x'')  \nonumber\\
&= \frac1{4\ell} \sum_{\substack{m,j=0\\ (x)}}^{4\ell-1}\t^{\frac{m(m-1)}2}
\frac{(\t-\t^{-1})^m}{[m]!} x'S(K^jF^m)\otimes \sum_{n=0}^{4\ell-1}\t^{\frac n2(2m+2|x''|-j)}\ad_{E^m}(x'')\ .  \nonumber
\end{align}
It belongs to $\bU$ because 
$\sum_{n=0}^{4\ell-1}\t^{\frac n2(2m+2|x''|-j)}$ vanishes when $j$ is odd. The other half of (c) can be calculated
similarly.

As for (d) we note that only the diagonal part $D$ of $R$ contributes to $\Gz$
possible elements outside of $\bU$. 
One can slide the diagonal parts to the same place using:
\begin{equation}\label{D}
D(x\otimes y) = (K^{2|y|}x\otimes y K^{2|x|}) D\ .
\end{equation}
Therefore the sliding only inserts even powers of $K$ in some places.
Since the Borromean ring has $0$ linking matrix the diagonal parts got
canceled after they are slided to the same place. Therefore $\Gz(B)$ is in $\bU$.
\end{proof}

\section{The 3-manifold invariants}
The WRT invariant and the Hennings invariant can be both 
calculated from $\Gz$ in Section \ref{bt}. Let $M$ be a closed 3-manifold and $T$ be an $m$-component
bottom tangle such that 
$M$ is the result of surgery on $\hat T$, the natural closure of $T$.
The Hennings invariant
\begin{equation}
\pz(M) = \frac{\lambda^{\otimes m}(\Gz(T))}{\lambda(\r^{-1})^{\sigma_+}\lambda(\r)^{\sigma_-}}\ ,
\end{equation}
where $\sigma_\pm$ is the number of positive/negative eigenvalues of the linking matrix of $\hat T$.
The discrepancy in sign is due to the fact that the value of $\Gz$ at the trivial bottom tangle with a positive
twist is $\r^{-1}$.

To define the WRT invariant we need to consider the representations of $\U$. 
For $1\le n\le \ell-1$ let $V_n$ be the irreducible representation of $\U$ of dimension $n$.
Recall that for any $\U$-module $V$ the quantum trace $\qt V: \U\to \BC$ is defined by
$\qt V(x) = \tr^V(K^2 x)$, where $\tr^V$ is the ordinary trace on $V$. The WRT invariant of $M$ is
$$
\tz(M) = \frac{(\qt\omega)^{\otimes m} (\Gz(T))}{\qt\omega(\r^{-1})^{\sigma_+}\qt\omega(\r)^{\sigma_-}}\ ,
$$
where $\qt\omega = \sum_{n=1}^{\ell-1} [n]\qt{V_n}$.

\section{Proof of the main theorem}\label{main}

The proof is divided into three cases:
\begin{enumerate}
	\item[(i)] $\h(M)=0$, i.e. $M$ has infinite first homology;
	\item[(ii)] $\h(M)\ne 0$ and $M$ is the result of surgery on a link with diagonal linking matrix;
	\item[(iii)] $\h(M)\ne0$ and $M$ can not be obtained by surgery on a link with diagonal linking matrix.
\end{enumerate}

Case (i) follows from the fact $\pz(M)=0$ if $M$ has infinite first homology according to \cite{o3, kerler2}.

Case (ii) will be proved in \ref{c2}.

To show (iii) we recall from \cite{o1} 
that in this case there exist lens spaces $L(n_i, 1)$, $i = 1\ldots m$ such that
$$
M' = M \# L(n_1,1) \# \cdots \# L(n_m,1)
$$
is the result of surgery on a link with diagonal linking matrix.
By (ii) we have
$\pz(M') = \h(M')\tau(M')$, which is the same as (because $\pz$, $\tz$ and $\h$ are multiplicative with respect to
connected sum):
$$
\pz(M)\prod_{i=1}^m \pz(L(n_i,1)) = \h(M) \tz(M) \prod_{i=1}^m |n_i|\tz(L(n_i,1))\ .
$$
It remains to note that $\tz(L(n_i,1))\ne 0$, c.f. \cite{lili2}, and 
$\pz(L(n_i,1)) = |n_i| \tz(L(n_i,1))$ by (ii). 

\subsection{Some lemmas about the center}\label{lc}
We will need some preparation lemmas that will lead to a proof of (ii) in \ref{c2}. Denote by
$\bbZ$ the subset of $\bZ$ spanned by $e_m$, $0\le m\le\ell$, and $w_n$, $1\le n < \ell$. The following lemma
follows from \cite{feigin} Proposition D.1.1.

\begin{lemma}\label{c}
$\bbZ = \BC[C]$ where
$$
C = FE + \frac{K^2\t+K^{-2}\t^{-1}}{(\t-\t^{-1})^2}\ .
$$
\end{lemma}

The next lemma deals with the value of the left integral $\lambda$, c.f. (\ref{lin}), on the center.

\begin{lemma}\label{ri}
The left integral $\lambda$ vanishes on $\bbZ$ and 
\begin{equation}\label{wn}
\lambda(w_m^+)= (-1)^{m-1}\t^{-2}\frac{[m]^3}{2\ell([\ell-1]!)^2}\ .
\end{equation}
\end{lemma}

\begin{proof} 
By Lemma \ref{c} the first half is equivalent to
\begin{equation}\label{ci}
\lambda(C^i) = 0, \quad i = 0,1,2,\ldots
\end{equation}
Because
$$
C^i = F^iE^i + \text{terms with lower degree of } E
$$
we see that (\ref{ci}) holds for $0\le i<\ell-1$.
It is known, c.f. (3.6) in \cite{feigin}, that
\begin{equation}\label{kl}
K^{2\ell} = \frac12\sum_{i=0}^{\lfloor\frac p2\rfloor} 
\frac{(-1)^{i-1}\ell}{\ell-i}\binom{\ell-i}i (\t-\t^{-1})^{2(\ell-2 i)} C^{\ell-2i} \ .
\end{equation}
This implies (\ref{ci}) by induction. Equation (\ref{wn}) follows from (4.19) in \cite{feigin}.
\end{proof}

Next we discuss the restriction of $\qt\omega$ on $\bbZ$.

\begin{lemma}\label{qt}
We have
\begin{equation}\label{eqt1}
\qt\omega(w_m^\pm) = \qt\omega(w_m) = 0, \quad 1\le m <\ell\ ,
\end{equation}
and
\begin{equation}\label{eqt2}
\qt\omega(e_m) = [m]^2\ .
\end{equation}
\end{lemma}

\begin{proof}
Recall that $V_n$ is the $n$-dimensional $\U$-module such that
$$C|_{V_n} = \frac{\t^n+\t^{-n}}{(\t-\t^{-1})^2}\  \Id|_{V_n}\ .$$
This lemma follows easily from the above equation and (D.3-5) in \cite{feigin}.
\end{proof}

It turns out that $\lambda$ and $\qt\omega$ are closely related:
\begin{lemma}\label{lr}
For any $x\in \bbZ$ and $n\in\BN$,
\begin{equation}\label{er}
\frac{\lambda(x\r^{\pm n})}{\lambda(\r^{\pm1})} = n\ \frac{\qt\omega(x\r^{\pm n})}{\qt\omega(\r^{\pm1})}
\end{equation}
\end{lemma}

\begin{proof}
From (\ref{eqc}, \ref{eqr}) we have
$$
\r^{\pm n} = \sum_{m=0}^\ell a^n_{\pm,m}e_m + n\, a^{n-1}_{\pm,m}\left(b_{\pm,m}w_m + c_{\pm,m}w^+_m\right)\ .
$$
We only need to prove (\ref{er}) for $x=e_j$ and $x=w_j$.
If $x = e_j$ then
\begin{align}\label{eej}
\frac{\lambda(x\r^{\pm n})}{\lambda(\r^{\pm1})} 
& = n\ \frac{a^{n-1}_{\pm,j}c_{\pm,j}\lambda(w^+_j)}{\sum_{m=0}^{\ell} c_{\pm,m}\lambda(w^+_m)} 
 = n\ \frac{(-1)^{(j-1)n}\t^{\pm\frac{(1-j^2)n}2}[j]^2}{\sum_{m=0}^{\ell}(-1)^{m-1}\t^{\pm\frac{1-m^2}2}[m]^2}
\nonumber \\
& = n\ \frac{a^n_{\pm,j}\qt\omega(e_j)}{\sum_{m=0}^\ell a_{\pm,m}\qt\omega(e_m)} = n\ \frac{\qt\omega(x\r^{\pm n})}{\qt\omega(\r^{\pm1})}\ .\nonumber
\end{align}

If $x = w_j$ then both sides of (\ref{er}) is 0.
\end{proof}

\subsection{An improvement of Lemma \ref{l1}}\label{li}
For any $A$-module $W$ of a Hopf algebra $A$ set
$$
\bar W = W/\{a x - \epsilon(a) x, \forall a\in A, x\in W\}\ . 
$$
It is clear that $\bar W$ inherits a trivial $A$-module structure from $W$.

Since $\U$ contains a cointegral $\Lambda$, $\lambda$ factors through $\bar\U$, c.f. Proposition 8 in \cite{larson}.
It is known that $\qt{V}$ also factors through
$\bar\U$ for any $\U$-module $V$, c.f. Section 7.2 in \cite{habiro1}.
We will also need the following lemma whose proof is immediate.

\begin{lemma}\label{lf}
For any $z\in\bZ$, $\lambda^z$ and $(\qt\omega)^z$
both factor through $\bbU$, where
$$
\lambda^z(x) := \lambda(xz), \quad\text{and}\quad (\qt\omega)^z(x) := \qt\omega(xz), \quad \forall x\in\bU\ .
$$
\end{lemma}

The following key lemma is an improvement of Lemma \ref{l1}. It is similar to Lemma 8 in \cite{cks}, which
was proved in a much more complicated way.

\begin{lemma}\label{l2}
Let $T$ be an $m$-component bottom tangle whose natural closure $\hat T$ has 0 linking matrix.
If $\chi_i: \bU\to \BC$ factors through $\bbU$ then 
$(\Id\otimes\chi_2\otimes\cdots\otimes\chi_m) \Gz(T)$ belongs to $\bbZ$.
\end{lemma}

\begin{proof}
We need another version of quantum $\fsl$. Let $\hU$ be the $\BC$-algebra generated by the same generators with
the same relations as $\U$ but omitting (\ref{ell}). Denote by $\p$ the canonical projection of $\hU$ to $\U$. 
The algebra $\hU$ itself is not quasi-triangular but there exists
in some completion of $\hU\otimes\hU$ a universal $R$-matrix
$$
\hat R = \hat D \sum_{n=0}^\infty \frac{(\t-\t^{-1})^n}{[n]!}\t^{\frac{n(n-1)}2} E^n\otimes F^n \ .
$$
Here $\hat D$ is the diagonal part, which satisfies the same relation as $D$ in (\ref{D}):
\begin{equation}\label{hD}
\hat D(x\otimes y) = (K^{2|y|}x\otimes y K^{2|x|}) \hat D\ .
\end{equation}
One can use Fig. \ref{f1} to define the universal invariant $\hat{\Gamma}_\zeta$ associated to $\hU$.
Since $\hat T$ has $0$ linking matrix one can cancel the diagonal parts by sliding them to the same place
using (\ref{hD}). Comparing the formulas of $R$ and $\hat R$ 
it is then clear that 
\begin{equation}\label{p}
\p\circ\hat{\Gamma}_\zeta(T) = \Gz(T)\ .
\end{equation}
The proof of Lemma \ref{l1} can be used word for word to show that $\hat{\Gamma}_\zeta(T)\in (\hUe)^{\otimes m}$.
Since $\bhUe$ inherits a trivial $\hUe$-module structure from the adjoint action and 
$\hat\chi_i:= \chi\circ\p$ factors through
$\bhUe$ we have
$$
(\Id\otimes\hat\chi_2\otimes\cdots\otimes\hat\chi_m) \hat{\Gamma}_\zeta(T) \in \Inv(\hUe) = \text{Center}(\hUe)\ .
$$
According to \cite{dk} Theorem 4.2, $\text{Center}(\hUe)$ is generated by $E^\ell$, $F^\ell$, $K^{\pm\ell}$
and $C$. Note that
their $K$ and $\epsilon$ are equal to our $K^2$ and $\t$ respectively. From (\ref{p}) we have
$$
(\Id\otimes\chi_2\otimes\cdots\otimes\chi_m) \Gz(T) = \p\circ
(\Id\otimes\hat\chi_2\otimes\cdots\otimes\hat\chi_m) \hat{\Gamma}_\zeta(T)\ ,
$$
which is a polynomial in $K^{\pm\ell}$ and $C$. It remains to note that $K^{\pm\ell}$ can be expressed as
a polynomial in $C$, c.f. (4.2.8) in \cite{dk} and (\ref{kl}).
\end{proof}

\subsection{Proof of (ii)}\label{c2}
Let $T$ be an $m$-component bottom tangle whose natural closure $\hat T$ has diagonal linking matrix
$\diag(f_1, f_2, \cdots, f_m)$. Suppose $M$ is the result of surgery on $\hat T$. Then
$\h(M) = |f_1\cdots f_m|$. Also assume that
$f_1, \ldots, f_i>0$ and $f_{i+1},\ldots, f_m < 0$.
Let $T_0$ be the bottom tangle obtained from $T$ by changing the framing on each component to 0.
We have
\begin{align}\label{last}
\pz(M) &= \frac{\lambda^{\otimes m}(\Gz(T))}{\lambda(\r^{-1})^{i}\lambda(\r)^{m-i}}
= \frac{\lambda^{\r^{-f_1}}\otimes\cdots\otimes\lambda^{\r^{f_m}}(\Gz(T_0))}{\lambda(\r^{-1})^{i}\lambda(\r)^{m-i}}
\nonumber\\ 
&= \frac{\lambda^{\r^{-f_1}}}{\lambda(\r^{-1})}\left( \frac{\Id\otimes\lambda^{\r^{-f_2}}\otimes\cdots\otimes\lambda^{\r^{f_m}}}
{\lambda(\r^{-1})^{i-1}\lambda(\r)^{m-i}}(\Gz(T_0))\right)
\nonumber\\ 
\left(\substack{\text{by Lemmas} \\ \text{\ref{lr} and \ref{l2}}}\right) &= 
|f_1|\ \frac{(\qt\omega)^{\r^{-f_1}}}{\qt\omega(\r^{-1})}\left( \frac{\Id\otimes\lambda^{\r^{-f_2}}\otimes\cdots\otimes\lambda^{\r^{f_m}}}
{\lambda(\r^{-1})^{i-1}\lambda(\r)^{m-i}}(\Gz(T_0))\right)
\nonumber\\
=\cdots &= |f_1\cdots f_m|\ \frac{(\qt\omega)^{\r^{-f_1}}\otimes\cdots\otimes(\qt\omega)^{\r^{f_m}}(\Gz(T_0))}
{\qt\omega(\r^{-1})^{i}\ \qt\omega(\r)^{m-i}}
\nonumber\\
& = \h(M)\ \tz(M)\ . \nonumber
\end{align}
This ends the proof of (ii) and hence the proof of Theorem \ref{mt}.


\end{document}